\documentclass{daj}

\dajAUTHORdetails{%
  title = {Schwartz-Zippel Bounds for Two-dimensional Products}, 
  author = {Hossein Nassajian Mojarrad, Thang Pham, Claudiu Valculescu, and Frank de Zeeuw},
  plaintextauthor = {Hossein Nassajian Mojarrad, Thang Pham, Claudiu Valculescu, Frank de Zeeuw},
  keywords = {combinatorial geometry, incidence bounds, distance problems},
}   

\dajEDITORdetails{%
   year={2017},
   number={20},
   received={19 September 2016},   
   published={20 December 2017},  
   doi={10.19086/da.2750},       
}   

\usepackage{amssymb,amsthm, amsmath}
\newtheorem{theorem}{Theorem}[section]
\newtheorem{corollary}[theorem]{Corollary}
\newtheorem{lemma}[theorem]{Lemma}

\newtheorem{question}[theorem]{Question}
\theoremstyle{definition}
\newtheorem{definition}[theorem]{Definition}
\newcommand{\C}{\mathbb C}
\newcommand{\R}{\mathbb R}
\newcommand{\pts}{\mathcal P}
\newcommand{\cvs}{\mathcal C}
\newcommand{\eps}{\varepsilon}

\begin{document}

\begin{frontmatter}[classification=text]

\title{Schwartz-Zippel Bounds for Two-dimensional Products
} 

\author[hnm]{Hossein Nassajian Mojarrad\thanks{Supported by Swiss National Science Foundation Grants 200020-165977 and 200021-162884.}}
\author[tp]{Thang Pham\footnotemark[1]}
\author[cv]{Claudiu Valculescu\footnotemark[1]}
\author[fdz]{Frank de Zeeuw\footnotemark[1]}

\begin{abstract}
We prove bounds on intersections of algebraic varieties in $\C^4$ with Cartesian products of finite sets from $\C^2$, and we point out connections with several classic theorems from combinatorial geometry.
Consider an algebraic variety $X\subset \C^4$ of degree $d$, 
such that not all polynomials that vanish on $X$ are of the form
\[ F(x,y,s,t) = G(x,y)H(x,y,s,t) + K(s,t)L(x,y,s,t),\]
where $G,H,K,L$ are polynomials and $G$ and $K$ are not constant.
Let $P,Q\subset \C^2$ be finite sets of size $n$.
If $X$ has dimension one or two, then we prove $|X\cap (P\times Q)| = O_d(n)$, while if $X$ has dimension three, then $|X\cap (P\times Q)| =O_{d,\eps}(n^{4/3+\eps})$ for any $\eps>0$.
Both bounds are best possible in this generality (except for the $\eps$).

These bounds can be viewed as different generalizations of the Schwartz-Zippel lemma, where we replace a product of ``one-dimensional'' finite subsets of $\C$ by a product of ``two-dimensional'' finite subsets of $\C^2$.
The bound for three-dimensional varieties generalizes the Szemer\'edi-Trotter theorem.
A key ingredient in our proofs is a two-dimensional version of a special case of Alon's combinatorial Nullstellensatz.

As corollaries of our two bounds,
we obtain bounds on the number of repeated and distinct values of polynomials and polynomial maps of pairs of points in $\C^2$, with a characterization of those maps for which no good bounds hold.
These results generalize known bounds on repeated and distinct Euclidean distances.
\end{abstract}
\end{frontmatter}


\section{Introduction}
\label{sec:intro}

A special case of the well-known Schwartz-Zippel\footnote{See \cite{L} for the curious history of this lemma. 
We use the name that has become standard in combinatorics, even though others published versions of this statement earlier, and many a nineteenth-century geometer could have proved the statement that we are referring to.}
lemma 
states that for an algebraic curve $C\subset \C^2$ of degree $d$ and two finite sets $A,B\subset \C$, we have\footnote{The notation $X=O_d(Y)$ (for positive quantities $X,Y$) means that there is a constant $C_d$, depending only on $d$, such that $X\leq C_d\cdot Y$. 
Similarly, $X =\Omega_d(Y)$ means that there is a $C_d>0$ such that $X\geq C_d\cdot Y$.}
\begin{align}\label{eq:schwartzzippel}
|C\cap (A\times B)|=O_d(|A|+|B|).
\end{align}
In other words, it bounds the size of the intersection of an algebraic curve with a Cartesian product of ``one-dimensional'' sets (let us informally call a finite set \emph{$k$-dimensional} if it is given as a subset of $\C^k$).
The full Schwartz-Zippel lemma generalizes this statement to polynomials over any field in any number of variables.
A related statement is Alon's combinatorial Nullstellensatz \cite{A}, which (in this case) says that $C$ contains all of  $A\times B$ if and only if it is defined by a polynomial of the form $g(x)h(x,y) + k(y)l(x,y)$, with $g$ vanishing on $A$ and $k$ vanishing on $B$.

In this paper we study generalizations of these statements to Cartesian products of \emph{two-dimensional} sets.
More precisely, given a variety\footnote{See \cite{CLOS, Shaf} for definitions of varieties and their dimension and degree. In this paper we only deal with \emph{affine} varieties over $\C$, defined as $Z(I)=\{(x_1,\ldots,x_m)\in \C^m:F(x_1,\ldots,x_m)=0~\forall F\in I\}$ for some $I\subset \C[x_1,\ldots,x_m]$;
we write $I(X) = \{F\in \C[x_1,\ldots,x_m]: F(x_1,\ldots,x_m)=0~\forall (x_1,\ldots,x_m)\in X\}$ for the ideal of polynomials defining $X$.
Note that a variety can be a union of components of different dimensions; then its dimension is the maximum of the dimensions of the components.} $X\subset \C^4$ and two finite sets $P,Q\subset \C^2$, we prove upper bounds on the size of the intersection
\begin{align}\label{eq:intersection}
|X\cap (P\times Q)|,
\end{align}
and we determine which $X$ can contain a whole product $P\times Q$.
Unlike in the one-dimensional case, we cannot expect a good bound on \eqref{eq:intersection} for all varieties. 
Take for instance $X = Z(F)$ with $F(x,y,s,t)=xs+yt$; if we take $P$ on the $y$-axis and $Q$ on the $s$-axis, then $X$ contains all of $P\times Q$.
This example is generalized in the following definition, which appears to be new.
Note the similarity with the special form in the version of Alon's Nullstellensatz mentioned above.

\begin{definition}\label{def:cartesian}
Given $G\in \C[x,y]\backslash\C$ and $K\in \C[s,t]\backslash\C$,
a polynomial $F\in \C[x,y,s,t]$ is $(G,K)$-\emph{Cartesian} if it can be written as
\begin{align}\label{eq:cartesian}
F(x,y,s,t) = G(x,y)H(x,y,s,t)+ K(s,t)L(x,y,s,t),
\end{align}
with $H,L\in \C[x,y,s,t] $;
$F$ is \emph{Cartesian} if there are non-constant $G,K$ such that $F$ is $(G,K)$-Cartesian.

Let $X$ be a variety in $\C^4$ and let $I(X)$ be its defining ideal.
Then $X$ is \emph{Cartesian} if there are $G\in \C[x,y]\backslash\C$, $K\in \C[s,t]\backslash\C$ such that every $F\in I(X)$ is $(G,K)$-Cartesian.
Note that a reducible variety is Cartesian if  one of its components is Cartesian, since the ideal of a union of varieties is the intersection of the ideals of the components.
\end{definition}


If $X$ is Cartesian, then we cannot give a better bound on \eqref{eq:intersection} than the trivial $|P||Q|$.
Indeed, given $G$ and $K$, we can arbitrarily choose finite subsets $P\subset Z(G)\subset \C^2$ and $Q\subset Z(K)\subset \C^2$.
Then for any polynomial $F$ of the form \eqref{eq:cartesian} we have $F(p,q) = 0$ for all $(p,q)\in P\times Q$, which means that $|X\cap (P\times Q)|=|P||Q|$.

Our two main theorems show that, as long as $X$ is not Cartesian, much better bounds hold.
Moreover, it turns out that there is a dichotomy between varieties of dimension three and varieties of dimension less than three\footnote{We will ignore the trivial cases of dimension four, for which the intersection size is always $|P||Q|$, and dimension zero, for which B\'ezout's inequality immediately gives the bound $O_d(1)$.}.
When $X$ has dimension one or two, our first main theorem gives a linear bound, which perhaps makes this statement the most natural generalization of the one-dimensional Schwartz-Zippel bound \eqref{eq:schwartzzippel}.

\begin{theorem}\label{thm:dimtwo}
Let $X$ be a variety in $\C^4$ of degree $d$ and dimension one or two,
and let $P,Q\subset \C^2$ be finite sets.
Then
\[|X\cap (P\times Q)| = O_d(|P|+|Q|), \]
unless $X$ is Cartesian.
\end{theorem}

Our second main theorem concerns varieties of dimension three, where we observe a very different bound.

\begin{theorem}\label{thm:dimthree}
Let $X$ be a variety in $\C^4$ of degree $d$ and dimension three,
and let $P,Q\subset \C^2$ be finite sets.
Then\footnote{Throughout this paper, when we state a bound involving an $\eps$, we mean that this bound holds for every positive $\eps\in\R$, with the multiplicative constant of the $O$-notation depending on $\eps$.}
\[ |X\cap (P\times Q)| =O_{d,\eps}\left( |P|^{2/3+\eps}|Q|^{2/3}+|P|+|Q|\right),\]
unless $X$ is Cartesian.
When $P,Q\subset \R^2$, the $\eps$ can be omitted.
\end{theorem}

The $\eps$ in Theorem \ref{thm:dimthree} could of course just as well be placed in the exponent of $|Q|$.
It comes from the incidence bound of Sheffer, Szab\'o, and Zahl \cite{SSZ} that we use.
We expect that this $\eps$ can be removed in general, but current techniques do not quite allow for this; see the discussion in Section \ref{sec:discussion}.
The statement holds for $P,Q\subset \R^2$ without the $\eps$, 
but it should be understood that the variety is still considered as a complex variety, and in particular we should take its complex dimension (which may differ from the real dimension of the corresponding real variety).
Similarly,
our definition of a polynomial $F$ being Cartesian is only formulated over $\C$, i.e., the polynomials $G,H,K,L$ could be complex polynomials even when $F$ is real.
With a little more work we could obtain a fully real statement,
but we currently do not see much benefit to that.

We note that many instances of Theorem \ref{thm:dimthree} have long been known.
Most notably, the Szemer\'edi-Trotter theorem \cite{ST}, which bounds the number of incidences between points and lines in $\R^2$, 
can be rephrased as the case $X=Z(F)$ for $F=xs-y+t$ in  Theorem \ref{thm:dimthree}; indeed, the point $(a,b)$ lies on the line $y=cx+d$ if and only if $F(a,b,c,d)=0$.
Another familiar example is the polynomial $F = (x-s)^2 + (y-t)^2 -1$, which we discuss in Section \ref{subsec:repeatedpoly}.
More generally, Theorem \ref{thm:dimthree} can be seen as a variant of the Pach-Sharir theorem \cite{PS}, and that theorem plays a crucial role in the proof; see the discussion in Section \ref{sec:discussion}.

Our proofs of Theorems \ref{thm:dimtwo} and \ref{thm:dimthree} rely on two-dimensional versions of the special case of Alon's Nullstellensatz mentioned at the start of this introduction.
Specifically, instead of asking when a two-variable polynomial vanishes on an entire product $A\times B\subset \C\times \C$,
we ask when a four-variable polynomial vanishes on an entire product $P\times Q\subset \C^2\times \C^2$.
These two-dimensional versions come in several forms, 
and to avoid overloading this introduction we state them as we prove them in Section \ref{sec:null}.
Alon's Nullstellensatz has many applications in combinatorics, but it is not yet clear if our two-dimensional generalizations have similar applications, other than the theorems in this paper.

We will give some applications of Theorems \ref{thm:dimtwo} and \ref{thm:dimthree} concerning repeated and distinct values of polynomials and polynomial maps.
We generalize well-known bounds on the squared Euclidean distance function $(x-s)^2+(y-t)^2$ to arbitrary polynomials $F(x,y,s,t)$, with exceptions related to the Cartesian form.
Again to avoid overloading the introduction, we present these applications and their background in Section \ref{sec:repeateddistinct}.


\paragraph{Constructions.}
Theorems \ref{thm:dimtwo} and \ref{thm:dimthree} are best possible for statements of this generality (except for the $\eps$ in the complex case of Theorem \ref{thm:dimthree}).
Let us make clear what we mean by this.

First we consider Theorem \ref{thm:dimtwo}.
Let $X$ be any variety in $\C^4$ of dimension one or two.
Define projections by $\pi_1(x,y,s,t) = (x,y)$ and $\pi_2(x,y,s,t) = (s,t)$.
For any $n$, we can choose a generic subset $R\subset X$ of size $n$ such that $|\pi_1(R)| = n$ and $|\pi_2(R)| = n$.
Then setting $P = \pi_1(R)$ and $Q = \pi_2(R)$ gives
 $|X\cap (P\times Q)| \geq |R| = n$.
 This shows that the bound of Theorem \ref{thm:dimtwo} is tight, in the sense that the exponents cannot be improved, at least when $|P|=|Q|$.
 When $|P|$ and $|Q|$ are far apart, this construction gives the lower bound $\min\{|P|,|Q|\}$, whereas the upper bound in Theorem \ref{thm:dimtwo} is $\max\{|P|,|Q|\}$.

 Theorem \ref{thm:dimthree} is tight in the following weaker sense: There are polynomials $F$ of any degree and point sets $P,Q$ for which the bound is tight (aside from the $\eps$).
This follows from a construction of Elekes \cite{El} (based on an earlier construction of Erd\H os) that shows the Szemer\'edi-Trotter bound to be tight.
Consider $F = xs -y+t$ (it is not hard to verify that this polynomial is not Cartesian).
For parameters $\lambda, \mu$, 
set 
\[P = \{(i,j):1\leq i\leq \lambda, ~1\leq j\leq \lambda \mu\},~~~Q = \{(i,j):1\leq i\leq \mu, ~1\leq j\leq \lambda \mu\}.\]
Then $|P|=\lambda^2\mu$ and $|Q| = \lambda\mu^2$, 
and all sizes $|P|,|Q|$ can be approximately obtained in this way.
Moreover, for $\Omega(\lambda^2\mu)$ of the points $(x,y)\in P$ there are $\Omega(\mu)$ points $(s,t)\in Q$ such that $xs-y+t = 0$, 
so we have $|Z(F)\cap (P\times Q)| = \Omega(\lambda^2\mu^2) =\Omega(|P|^{2/3}|Q|^{2/3})$.


This construction is easily extended to polynomials of any degree; 
for instance, take the polynomial $F=xs+y^d-t^d$
and points of the form $(i,j^{1/d})$.
Another example of a tight construction is given by Valtr \cite{V}, for the polynomial $F = (x-s)^2+y-t$.

On the other hand, it is conjectured that for $F = (x-s)^2 + (y-t)^2 -1$ (see also Section \ref{sec:repeateddistinct}),
the bound $|Z(F)\cap (P\times P)| = O(|P|^{4/3})$ is not tight.
Erd\H os \cite{Er} conjectured the bound $O_\eps(|P|^{1+\eps})$, but no better bound than $O(|P|^{4/3})$ is known for this or any other polynomial.
It would be interesting to find out the distinction between the polynomials for which the bound is tight, and those for which it is not tight.

\paragraph{Outline.}
In Section \ref{sec:null}, we prove our variants of Alon's Nullstellensatz for two-dimensional products, which are key tools in our later proofs.
In Section \ref{sec:dimtwo} we prove Theorem \ref{thm:dimtwo} and in Section \ref{sec:dimthree} we prove Theorem \ref{thm:dimthree}.
The applications to repeated and distinct values of polynomials and polynomial maps follow in Section \ref{sec:repeateddistinct}.
Finally, in Section \ref{sec:discussion}, we discuss some related results and possible extensions.


\section{Nullstellens\"atze for two-dimensional products}
\label{sec:null}

\subsection{First Nullstellensatz for two-dimensional products}

Here is the form of Alon's combinatorial Nullstellensatz that we will prove a variant of.
We say that a polynomial is \emph{squarefree} if it has no repeated factors.

\begin{theorem}\label{thm:alon}
Let $f\in \C[x,y]$, and let $g\in \C[x],k\in\C[y]$ be squarefree.
Then $Z(g)\times Z(k)\subset Z(f)$ if and only if there are $h,l\in \C[x,y]$ such that $f = g(x)h(x,y)+k(y)l(x,y)$.
\end{theorem}

This statement is a special (but crucial) case of \cite[Theorem 1.1]{A}; the general theorem applies to any field and any number of variables, and comes with bounds on the degrees of $h$ and $l$.
Most applications of the combinatorial Nullstellensatz use \cite[Theorem 1.2]{A}, which is a consequence of \cite[Theorem 1.1]{A}; 
the key observation behind this consequence is that one can tell from a single coefficient of $f$ that $f$ does not have the special form.

The combinatorial Nullstellensatz is of course named after Hilbert's Nullstellensatz (see for instance \cite{CLOS, Shaf}), 
which states that if a polynomial $f$ vanishes on the zero set $Z(I)$ of an ideal $I$, 
then some power of $f$ lies in $I$.
In other words, if $I=(g_1,\ldots, g_\ell)$, 
then $Z(I)\subset Z(f)$ implies $f^m = \sum g_ih_i$ for some integer $m$ and polynomials $h_i$.
Theorem \ref{thm:alon} thus gives the slightly more precise information that we can take $m=1$ in the case where $I = (g,k)$ and $g,k$ are univariate polynomials in different variables.

Our first Nullstellensatz is an analogue of Theorem \ref{thm:alon} for two-dimensional products.

\begin{theorem}\label{thm:null1}
Let $F\in \C[x,y,s,t]$, and let $G\in \C[x,y]\backslash\C,K\in \C[s,t]\backslash\C$ be squarefree.
Then we have  $ Z(G)\times Z(K)\subset Z(F)$ if and only if $F$ is $(G,K)$-Cartesian.
\end{theorem}
\begin{proof}
One direction is obvious: If $F$ is $(G,K)$-Cartesian, then $Z(G)\times Z(K)\subset Z(F)$.

For the other direction, we do the following, using terminology from \cite[Chapter 2]{CLOS}.
Fix an arbitrary monomial ordering on the variables $x,y,s,t$, and apply the multivariate division algorithm (\cite[Theorem 2.3.3]{CLOS}) to the polynomials $F(x,y,s,t)$ and $G(x,y)$.
The division algorithm tells us that there exist polynomials $H(x,y,s,t)$ and $R(x,y,s,t)$ such that
\begin{equation}\label{eq:division}
F(x,y,s,t)=G(x,y)\cdot H(x,y,s,t) + R(x,y,s,t),
\end{equation}
and no monomial of $R(x,y,s,t)$ is divisible by the leading monomial of $G(x,y)$.

For a fixed $q =(s_q,t_q)\in Z(K)$, 
the fact that $Z(G)\times Z(K)\subset Z(F)$ implies that $Z(G)$ is contained in the zero set of $F(x,y,s_q,t_q)$.
Because $G$ is squarefree, it follows from Hilbert's Nullstellensatz that $G(x,y)$ divides $F(x,y,s_q,t_q)$, 
so by \eqref{eq:division} it also divides $R(x,y,s_q,t_q)$.
Suppose that $R(x,y,s_q,t_q)\not\equiv 0$; then the leading monomial of $G(x,y)$ divides some monomial of $R(x,y,s_q,t_q)$.
But a monomial of $R(x,y,s_q,t_q)$ necessarily divides some monomial of $R(x,y,s,t)$, so we would have the leading monomial of $G(x,y)$ dividing a monomial of $R(x,y,s,t)$, contradicting the stated property of $R$.
Thus we have $R(x,y,s_q,t_q)\equiv 0$ for every $q\in Z(K)$.

If we expand $R(x,y,s,t)$ as
\[R(x,y,s,t)=\sum_{i,j}R_{ij}(s,t) x^iy^j,\]
then $R(x,y,s_q,t_q)\equiv 0$ implies that $R_{ij}(s_q,t_q)=0$ for all $i,j$.
Hence we have $Z(K)\subset Z(R_{ij})$ for all $i,j$.
Since $K$ is squarefree, it follows that $K$ divides $R_{ij}$, so for each pair $i,j$ there exists $L_{ij}\in \C[s,t]$ such that  $R_{ij}(s,t) = K(s,t)L_{ij}(s,t)$.
Then  we can write
$$F(x,y,s,t)=G(x,y)\cdot H(x,y,s,t) + K(s,t) \cdot \left(\sum_{i,j}L_{ij}(s,t) x^iy^j\right),$$
so setting $L=\sum L_{ij}(s,t)x^iy^j$, we see that $F$ is $(G,K)$-Cartesian.
\end{proof}

Although Theorem \ref{thm:null1} is a direct generalization of Theorem \ref{thm:alon}, there is one essential difference.
The zero sets $Z(g)$ and $Z(k)$ in Theorem \ref{thm:alon} are \emph{finite sets}, while $Z(G)$ and $Z(K)$ in Theorem \ref{thm:null1} are \emph{algebraic curves}. 
Our second Nullstellensatz (Theorem \ref{thm:null2} below)
concerns finite subsets of $\C^2$, 
and is thus perhaps closer in spirit to Theorem \ref{thm:alon}.

\subsection{Curves and dual curves}

We require the following terminology of curves\footnote{We define a \emph{curve} in $\C^2$ to be the zero set of any polynomial in $\C[x,y]\backslash\C$.
Note that a one-dimensional variety need not be a curve, because it may have zero-dimensional components, which cannot be described as the zero set of a single polynomial.
The \emph{degree} of a curve is the minimum degree of a defining polynomial; a curve of degree $d$ is a union of at most $d$ irreducible curves (its \emph{irreducible components}).} and dual curves, which we will also use in later proofs.
Given $F\in\C[x,y,s,t]$ and a point $q=(s_q,t_q)\in \C^2$, we define an algebraic curve in $\C^2$ by 
\[C_q 
= \{(x,y)\in\C^2:F(x,y,s_q,t_q)=0\}.\]

Note that it is possible that $C_q$ is not a curve but equal to $\C^2$; take for instance $F=xs+yt$ and $(s_q,t_q) = (0,0)$.
Fortunately, this cannot happen often if $F$ is not Cartesian, as we show in Lemma \ref{lem:degenerate} below.
We will abuse terminology somewhat and always refer to $C_q$ as a ``curve'', although of course we will be careful about this during our proofs.

To help us analyze the curves $C_q$, we define a ``dual'' curve for each $p=(x_p,y_p)\in \C^2$ by
\[C_p^* 
= \{(s,t)\in\C^2:F(x_p,y_p,s,t)=0\}.\]
The curves are dual in the sense that $p\in C_q$ if and only if $q\in C_p^*$ (and this is still true if one of the sets equals $\C^2$).

To study these curves we need the following slight generalization of B\'ezout's inequality.
The difference with the usual B\'ezout's inequality is that this statement applies to arbitrary collections of curves, 
and that it gives a lower bound on $|C_0\cap I|$.

\begin{lemma}
\label{lem:bezoutmany}
Let $\mathcal{S}$ be a (possibly infinite) set of curves in $\C^2$ of degree at most $d$, 
and suppose that their intersection $\cap_{C\in \mathcal{S}} C$ contains a set $I$ of size $|I|>d^2$.
Then there is a curve $C_0$ such that $C_0\subset \cap_{C\in \mathcal{S}} C$ and $|C_0\cap I|\geq |I|-(d-1)^2$.
\end{lemma}
\begin{proof}
The fact that there is a curve contained in $\cap_{C\in \mathcal{S}} C$ is proved in \cite[Lemma 3.10]{RSZ}.
It remains to be proved that there is such a curve  $C_0$ with $|C_0\cap I|\geq |I|-(d-1)^2$.

We can choose $C_0$ to be the union of all curves contained in $\cap_{C\in \mathcal{S}} C$.
For each $C\in \mathcal{S}$, we let $C'$ be the curve that remains after removal of $C_0$; in other words, it is the Zariski closure of $C\backslash C_0$ (which could be empty, but that would imply $I\subset C_0$ and we would be done).
Let $\mathcal{S}'$ be the set of all such curves $C'$.
Each curve $C'\in \mathcal{S}'$ has degree at most $d-1$, since we removed a curve of degree at least one.
Applying the first part of the lemma to $\mathcal{S}'$, 
we get that $|\cap_{C'\in \mathcal{S}'}C'| \leq (d-1)^2$, 
which implies $|C_0\cap I|\geq |I|-(d-1)^2$.
\end{proof}

With this lemma, we can prove that there are not many points for which $C_q$ is not a curve, unless $F$ is Cartesian.
A very similar fact is proved in \cite{RSZ15}.

\begin{lemma}\label{lem:degenerate}
Let $F\in \C[x,y,s,t]$ be a polynomial that is not Cartesian.
Then there are at most $d^2$ points $q\in \C^2$ for which $C_q = \C^2$.
\end{lemma}
\begin{proof}
We can expand $F$ as $\sum F_{ij}(s,t)x^iy^j$.
If for $q = (s_q,t_q)$ we have $C_q=\C^2$, then we must have $F_{ij}(s_q,t_q) = 0$ for all $i,j$.
If we set $I = \{q\in \C^2: C_q = \C^2\}$,
then we have $I\subset \cap_{i,j} Z(F_{ij})$.
By Lemma \ref{lem:bezoutmany},
either we have $|I|\leq d^2$, 
or the curves $Z(F_{ij})$ have a common curve, 
which means that the polynomials $F_{ij}(s_q,t_q)$ have a nontrivial common factor $K(s,t)$.
In the second case, $F$ would be Cartesian. 
To be precise, it would have the form $K(s,t)L(x,y,s,t)$, 
so it would be $(G,K)$-Cartesian for any $G$, by choosing $H=0$.
\end{proof}

\subsection{Second Nullstellensatz for two-dimensional products}

We can now improve on Theorem \ref{thm:null1} by using the duality between the curves $C_q$ and the curves $C_p^*$ to show that $F$ being Cartesian is equivalent to $Z(F)$ containing a sufficiently large finite product $I\times J$.
This statement will play a crucial role in the proofs of our main theorems.
The proof is similar to that of \cite[Lemma 3.11]{RSZ}.

\begin{theorem}\label{thm:null2}
A polynomial $F\in \C[x,y,s,t]$ of degree $d$
is Cartesian if and only if there are $I,J\subset \C^2$ with $|I|,|J|>d^2$ such that $I\times J\subset Z(F)$.
\end{theorem}
\begin{proof}
If $F$ is $(G,K)$-Cartesian, then for any $I\subset Z(G)$ and $J\subset Z(K)$ we have $I\times J\subset Z(F)$.
Suppose that $F$ is not Cartesian, and that $I\times J\subset Z(F)$ and $|I|,|J|>d^2$.
Then for all $p\in I$ and $q\in J$ we have $F(p,q)=0$, or in other words $p\in C_q$.
So $I\subset C_q$ for all $q\in J$.

Let $J_1$ be the set of $q\in J$ for which $C_q$ is a curve, and let $J_2$ be the set of $q\in J$ for which $C_q = \C^2$.
By Lemma \ref{lem:degenerate}, the assumption that $F$ is not Cartesian, and the fact that $|J|>d^2$, the set $J_1$ is not empty.
The curves $C_q$ for $q\in J_1$ have degree at most $d$ and we have $|I|>d^2$, so by Lemma \ref{lem:bezoutmany}, there is a curve $C$ such that $C\subset C_q$ for all $q\in J_1$.
We trivially have $C\subset C_q$ for all $q\in J_2$, so we can say that $C\subset C_q$ for all $q\in J$.

For $p\in C$ and $q\in J$ we have $p\in C_q$, so by duality $q\in C^*_p$; in other words, we have $J\subset C^*_p$ for all $p\in C$.
Again by Lemma \ref{lem:bezoutmany} and Lemma \ref{lem:degenerate},
it follows that there is a curve $C^*$ such that $C^*\subset C^*_p$ for all $p\in C$.
Thus for all $p\in C$ and all $q\in C^*$ we have $q\in C_p^*$, i.e., $F(p,q)=0$.
Therefore $C\times C^*\subset Z(F)$, which by Theorem \ref{thm:null1} is equivalent to $F$ being Cartesian.
\end{proof}

As we said, Theorem \ref{thm:null2} is closer to Theorem \ref{thm:alon} than Theorem \ref{thm:null1} is, because it concerns finite sets.
However, it differs in another way: The connection between the Cartesian product $I\times J$ and the $G,K$ in the Cartesian form is less clear.
It is not necessarily the case that if $I\times J\subset Z(F)$, then $F$ is $(G,K)$-Cartesian with $I\subset Z(G)$ and $J\subset Z(K)$.
The best we can say is the following refinement of Theorem \ref{thm:null2}, which follows directly by using the last claim of Lemma \ref{lem:bezoutmany} in the proof of Theorem \ref{thm:null2}.

\begin{corollary}\label{cor:null2}
Let $F\in \C[x,y,s,t]$ have degree $d$.
If $I\times J\subset Z(F)$ for $I,J\subset \C^2$ with $|I|,|J|>d^2$, 
then there are $G\in \C[x,y]\backslash\C,K\in\C[s,t]\backslash\C$ of degree at most $d$ such that $F$ is $(G,K)$-Cartesian, 
 and moreover 
 \[|I\cap Z(G)|\geq |I|-(d-1)^2~~~\text{and}~~~|J\cap Z(K)|\geq |J|-(d-1)^2.\]
\end{corollary}

Let us discuss a bit further how Theorem \ref{thm:null1} differs from Alon's Nullstellensatz.
As remarked at the start of this section,
most applications of Theorem \ref{thm:alon} use the fact that one can see from a single coefficient of $f$ that $f$ does not have the special form $g(x)h(x,y)+k(y)l(x,y)$.
This requires the more detailed information, given in \cite[Theorem 1.1]{A}, that $\deg(h)\leq \deg(f) - \deg(g)$ and $\deg(l)\leq \deg(f) - \deg(k)$.
This implies that any leading terms of $g(x)h(x,y)+k(y)l(x,y)$ are divisible by $x^{\deg(g)}$ or by $y^{\deg(k)}$.
If $f$ has a leading term $x^{\deg(g)-1}y^{\deg(k)-1}$ (say) with a nonzero coefficient, then $f$ cannot have the special form.

In the proof of Theorem \ref{thm:null1}, we could also obtain information on the degrees of $H$ and $L$.
In fact, in whatever monomial ordering we do the multivariate division \eqref{eq:division},
the multidegree (see \cite[Chapter 2]{CLOS}) of $H$ is at most the multidegree of $F$ minus that of $G$.
If we choose an ordering that respects the total degree of the polynomials, then we get $\deg(H)\leq \deg(F) - \deg(G)$, and similarly $\deg(L)\leq \deg(F) - \deg(K)$.

However, unlike in the one-dimensional case, 
we cannot say much about the degree of $G$ or $K$, which may be as small as one.
Indeed, we could have $I\times J\subset Z(F)$ with $I$ and $J$ each contained in a line, so that $G$ and $K$ are linear polynomials.
Then it seems hard to deduce anything about the coefficients of $G(x,y)H(x,y,s,t)+ K(s,t)L(x,y,s,t)$.


\section{Varieties of dimension one and two}
\label{sec:dimtwo}

Before proving Theorem \ref{thm:dimtwo}, we establish the following intermediate version, which is a Schwartz-Zippel bound for a Cartesian product of one-dimensional sets similar to that in \eqref{eq:schwartzzippel} in Section \ref{sec:intro}, but now the one-dimensional sets are finite subsets of algebraic curves.

\begin{lemma}
\label{lem:curvecurve}
Let $G\in \C[x,y]\backslash\C, K\in \C[s,t]\backslash\C$ be irreducible polynomials of degree at most $\delta$, 
let $F\in \C[x,y,s,t]\backslash \{0\}$ be a polynomial of degree $d$, 
and let $P\subset Z(G),Q\subset Z(K)$ be finite sets.
Then
\[|Z(F)\cap (P\times Q)| = O_{d,\delta}(|P|+|Q|),\]
unless $F$ is $(G,K)$-Cartesian.
\end{lemma}
\begin{proof}
Recall that $C_q=\{p\in \C^2:F(p,q)=0\}$.
We have
\[|Z(F)\cap (P\times Q)| = \sum_{q \in Q} |C_q\cap P|,\]
so it suffices to bound $|C_q \cap P|$ for each $q\in Q$. 
We set $Q_1 = \{q\in Q: Z(G)\subset C_q\}$ and $Q_2 = Q\backslash Q_1$.
Note that if $C_q = \C^2$, then $q\in Q_1$, and this will not be a problem.
For $q \in Q_2$ we have $|C_q\cap P|\leq |C_q \cap Z(G)|\leq d\cdot \delta$ by B\'{e}zout's inequality (see for instance Lemma \ref{lem:bezoutmany}) and the fact that $G$ is irreducible.
As a result we have 
\[\sum_{q \in Q_2} |C_q \cap P|\le \sum_{q \in Q_2} |C_q\cap Z(G)| \le d\cdot \delta\cdot |Q_2|  = O_{d,\delta}(|Q|).\]
If $|Q_1|\leq 2d^2$, then we can use the trivial bound $|C_q\cap P|\leq |P|$ to get 
\[\sum_{q \in Q_1} |C_q\cap  P| \le 2d^2\cdot |P| = O_d(|P|).\]

Otherwise, we can set $J=Q_1\subset Z(K)$ and let $I$ be any subset of $Z(G)$ with $2d^2$ elements.
Then we have $I\times J\subset Z(F)$ and $|I|,|J|>d^2$, so Corollary \ref{cor:null2} tells us that $F$ is $(G',K')$-Cartesian for some $G',K'$ of degree at most $d$, 
and moreover we have $|I\cap Z(G')|\ge |I|-(d-1)^2 > d^2$ and 
$|J\cap Z(K')|\ge|I|-(d-1)^2 > d^2$.
Thus $|Z(G)\cap Z(G')|>d^2$, 
so by B\'ezout's inequality and the fact that $G$ is irreducible, $G$ divides $G'$, and similarly $K$ divides $K'$.
Hence $F$ is $(G,K)$-Cartesian.
\end{proof}

As an aside, we obtain a bound on the number of repeated values of a polynomial on an algebraic curve.
This is a one-dimensional version of the question that we consider in Section \ref{subsec:repeatedpoly}.
See \cite{VZ} for a discussion of the related question on the number of distinct values on curves.

\begin{corollary}[Repeated values of polynomials on curves]\label{cor:repvaloncurves}
Let $C\subset \C^2$ be an algebraic curve of degree $\delta$ and let $F\in \C[x,y,s,t]\backslash\{0\}$ be a polynomial of degree $d$.
Then for any $a\in \C$ and finite set $P\subset C$,
the number of times $F$ takes the value $a$ on $P\times P$ satisfies
\[\left|\{(x,y,s,t)\in P\times P: F(x,y,s,t) = a\}\right| = O_{\delta,d}(|P|),\]
unless $F-a$ is $(G,G)$-Cartesian (i.e., $G$ divides $F-a$) for a polynomial $G\in \C[u,v]$ such that $C = Z(G)$.
\end{corollary}

Before we get back to varieties in $\C^4$ of dimension at most two, we need the following technical lemma from algebraic geometry.
A qualitative version of the lemma can be found for instance in \cite[Theorem I.6.3.7]{Shaf} (where it is called ``The theorem on the dimension of the fibers'').
However, we need a quantitative version where the dependence on the degree of the variety is specified. 
Such a statement is proved in \cite[Lemma 3.7]{BGT}, 
although the argument there is very general and the dependence on $d$ is not given explicitly.
In this particular case, it would not be hard to show by a direct argument that $W$ has degree $O(d^2)$.

\begin{lemma}\label{lem:fibers}
Let $\pi:\C^4\to \C^2$ be the projection defined by $\pi(x,y,s,t) = (s,t)$.
Let $X\subset \C^4$ be a variety of degree $d$ and dimension at most two.
Then there is a curve $W\subset \C^2$ of degree $O_d(1)$ such that, for every $(s,t)\not\in W$, $\pi^{-1}(s,t)\cap X$ is a finite set of size at most $d$.
\end{lemma}

We are now ready to bound $|X\cap (P\times Q)|$ when $X$ has dimension at most two.

\begin{proof}[Proof of Theorem \ref{thm:dimtwo}.]
Let $\pi_1,\pi_2$ be the projections defined by $\pi_1(x,y,s,t)=(x,y)$ and $\pi_2(x,y,s,t)=(s,t)$. 
By  Lemma \ref{lem:fibers}, there is a curve $W\subset \mathbb{C}^2$ of degree $O_d(1)$ such that for any $q\notin W$ we have $|\pi_2^{-1}(q)\cap X| \leq d$.
Similarly, there is a curve $V\subset \mathbb{C}^2$ of degree $O_d(1)$ such that for any $p\notin V$ we have $|\pi_1^{-1}(p)\cap X| \leq d$.

Let $P_1 = P\backslash V$ and $P_2 = P\cap V$, and similarly $Q_1 = Q\backslash W$ and $Q_2 = Q\cap W$.
Then
\begin{align}\label{eq:QminusW}
|X\cap (P_1\times Q)|\leq\sum_{p\in P_1} |\pi_1^{-1}(p)\cap X| \leq d\cdot |P_1| = O_d( |P|),
\end{align}
and similarly $|X\cap (P\times Q_1)|=O_d( |Q|)$.

It remains to bound $|X\cap (P_2\times Q_2)|$.
Let $Z(G_1),\ldots, Z(G_v)$ be the irreducible components of $V$, and let $Z(K_1),\ldots, Z(K_w)$ be the irreducible components of $W$.
Note that $v,w = O_d(1)$.
We have
\begin{align}\label{eq:sumofpairs}
|X\cap (P_2\times Q_2)| \leq \sum_{i,j}|X\cap ((P_2\cap Z(G_i))\times (Q_2\cap Z(K_j))| .
\end{align}

Consider any pair $G_i, K_j$.
By Lemma \ref{lem:curvecurve}, for each $F\in I(X)$ we have
\begin{align}\label{eq:VWcomp}
 |Z(F)\cap ((P_2\cap Z(G_i))\times (Q_2\cap Z(K_j))| = O_d(|P|+|Q|),
\end{align}
unless $F$ is $(G_i,K_j)$-Cartesian.
In other words, if some $F\in I(X)$ is not $(G_i,K_j)$-Cartesian, then we have this bound, or else all $F\in I(X)$ are $(G_i,K_j)$-Cartesian, which implies that $X$ is Cartesian.
Thus, if $X$ is not Cartesian, then for every pair $i,j$ we have the bound \eqref{eq:VWcomp}.
Summing the bound in \eqref{eq:VWcomp} over the $O_d(1)$ pairs $i,j$ as in \eqref{eq:sumofpairs}, and combining with \eqref{eq:QminusW}, we obtain the stated bound.
\end{proof}

We note that in Theorem \ref{thm:dimtwo}, the dependence of the bound on $d$ could be determined from our proof, if one makes the degree $O_d(1)$ of $W$ in Lemma \ref{lem:fibers} explicit.
For Theorem \ref{thm:dimthree} this would be more difficult, because its proof relies on Theorem \ref{thm:shefferzahl} below, in which the degree dependence is not made explicit.


\section{Varieties of dimension three}
\label{sec:dimthree}

We now consider three-dimensional varieties $X\subset \C^4$.
We reduce our problem of bounding $|X\cap (P\times Q)|$ to the problem of bounding incidences between points and curves, and then apply a well-known incidence bound.
To show that the conditions of this bound are met, we again use the curve duality introduced in Section \ref{sec:null}.
This duality was similarly used in combination with incidence bounds in a number of recent papers, including \cite{PZ, RSZ, SSS}.
The challenge in our situation is to convert the combinatorial condition of the incidence bound to an algebraic condition on the polynomials defining $X$.

Our main tool is the following incidence bound.
The statement was proved by Pach and Sharir \cite{PS} in $\R^2$, and this was generalized to $\C^2$ by Sheffer, Szab\'o, and Zahl \cite{SSZ}, at the cost of an $\eps$.
The \emph{incidence graph} defined by a set $\pts$ of points and a set $\cvs$ of curves is the bipartite graph with vertex sets $\pts$ and $\cvs$, with an edge between $p\in \pts$ and $C\in \cvs$ if $p\in C$.
We denote the number of edges in this graph by $I(\pts,\cvs)$; in other words, $I(\pts,\cvs)$ is the number of \emph{incidences} between $\pts$ and $\cvs$.
We denote by $K_{s,t}$ the complete bipartite graph on $s$ and $t$ vertices.

\begin{theorem}
\label{thm:shefferzahl}
Let $\pts\subset \C^2$ and let $\mathcal{C}$ be a set of algebraic curves of degree at most $d$.
If the incidence graph of $ \pts$ and $ \mathcal{C}$ contains no $K_{2,M}$ or $K_{M,2}$,
then
\[I(\pts,\mathcal{C}) = O_{d,M,\eps}\left(|\pts|^{2/3+\eps}|\mathcal{C}|^{2/3}+|\pts|+|\mathcal{C}|\right).\]
If $\pts\subset \R^2$, then the $\eps$ can be omitted.
\end{theorem}

In Theorem \ref{thm:dimthree} we are given a polynomial $F\in \C[x,y,s,t]$ and finite sets $P,Q\subset\C^2$.
Recall from Section \ref{sec:null} that for each $q=(s_q,t_q)\in Q$ we define a curve
\[C_q 
= \{(x,y)\in\C^2:F(x,y,s_q,t_q)=0\},\]
and for each $p=(x_p,y_p)\in P$ we define a dual curve
\[C_p^* 
= \{(s,t)\in\C^2:F(x_p,y_p,s,t)=0\}.\]
Let $\cvs_Q=\{C_q: q\in Q\}$ be the \emph{multiset} of curves defined by $Q$.

Note that we have to be careful when we pass from the point set $Q$ to the curves $C_q$,
because two different points $q,q'\in Q$ may define the same curve $C_q = C_{q'}$ (viewed as sets).
For this reason we deal with $\cvs_Q$ as a multiset, and we extend the notion of incidence graph to multisets in the obvious way (distinct $q,q'$ give distinct vertices, even if $C_q = C_{q'}$).
Although Theorem \ref{thm:shefferzahl} is not stated for multisets of curves, our Corollary \ref{cor:noKMM} below is.

We now show that for points and curves defined in this way, we can use the duality to obtain a stronger version of Theorem \ref{thm:shefferzahl} for our situation. 
Specifically, the condition on the excluded complete bipartite graph is much weaker.
A similar argument was used in \cite[Lemma 3.4]{PZ}.
A comparable statement was proved in a very different way by Fox et al. \cite{FPSSZ}; see also the discussion in Section \ref{sec:discussion}.

\begin{corollary}\label{cor:noKMM}
Let $P,Q\subset \C^2$ and let $F\in \C[x,y,s,t]$ have degree $d$.
Let $\mathcal{C}_Q$ be the multiset of curves defined by $Q$ and $F$.
If the incidence graph of $P$ and $ \mathcal{C}_Q$ contains no $K_{M,M}$, then
\[I(P,\mathcal{C}_Q) = O_{d,M,\eps}\left(|P|^{2/3+\eps}|Q|^{2/3}+|P|+|Q|\right).\]
If $P,Q\subset \R^2$, then the $\eps$ can be omitted.
\end{corollary}
\begin{proof}
First note that we can assume that $M>d^2$, since if $M\leq d^2$ and the incidence graph contains no $K_{M,M}$, then it also contains no $K_{d^2+1,d^2+1}$, so we can replace $M$ by $d^2+1$.

We claim that we can partition $P$ into $O_{d,M}(1)$ subsets $P_i$, and $Q$ into $O_{d,M}(1)$ subsets $Q_j$, so that the following holds.
Each $\cvs_{Q_j}$ is a \emph{set}, 
and for each pair $i,j$, the incidence graph of $P_i$ and $\cvs_{Q_j}$ contains no $K_{2,2dM}$ or $K_{2dM,2}$.
Given this claim, we can apply Theorem \ref{thm:shefferzahl} to the points $P_i$ and the curves $\cvs_{Q_j}$ to get
\[I(P_i,\cvs_{Q_j}) = O_{d,M,\eps}\left(|P_i|^{2/3+\eps}|Q_j|^{2/3}+|P_i|+|Q_j|\right). \]
Summing up these bounds for the $O_{d,M}(1)$ pairs $i,j$ proves the corollary.

The rest of the proof concerns the claim that such a partition exists.
Construct a graph $G$ with vertex set $P$ and an edge between two points $p,p'\in P$ if there are at least $2dM$ points $q\in Q$ such that $p,p'\in C_q$.
By duality, 
if $G$ has an edge between $p$ and $p'$,
then the intersection of the curves $C_p^*$ and $C_{p'}^*$ contains a set $I$ of at least $2dM$ points of $Q$.
Since $2dM>d^2$, Lemma \ref{lem:bezoutmany} implies that $C_p^*$ and $C_{p'}^*$ contain a common curve $C_0$, and moreover $|C_0\cap I|\geq |I| - (d-1)^2 > dM$.
This curve has degree at most $d$ and thus at most $d$ irreducible components, one of which contains at least $M$ points of $Q$.
To summarize, an edge between $p$ and $p'$ implies the existence of a common irreducible component of $C_p^*$ and $C_{p'}^*$ that contains at least $M$ points of $Q$.

We can use this observation to bound the maximum degree of a vertex of $G$ by $dM$.
Suppose that a vertex $p\in P$ has degree greater than $dM$.
Since $C_p^*$ has at most $d$ irreducible components, there must be an irreducible component $C_1$ of $C_p^*$ that contains at least $M$ points of $Q$, and that is shared with the curve $C_{p'}^*$ for at least $M$ neighbors $p'$ of $p$.
Then there  would be a $K_{M,M}$ in the incidence graph of $P$ and $\cvs_Q$, with on one side these $M$ neighbors $p'$, and on the other side any $M$ points of $Q$ on $C_1$.
This contradicts the assumption, so the degree of a vertex of $G$ is at most $dM$.

By basic graph theory\footnote{See for instance \cite[Chapter 14]{BM}; the simple proof is based on a greedy algorithm.},
 we can color $G$ with $dM+1$ colors, or in other words, we can partition $P$ into $dM+1 = O_{d,M}(1)$ subsets $P_i$ such that no two vertices in the same subset are adjacent in $G$.
 Since an edge in $G$ corresponds to a $K_{2,2dM}$ in the incidence graph, 
 this means that the incidence graph of $P_i$ and $\cvs_Q$ contains no $K_{2,2dM}$.

Let $Q_0$ be the subset of $q\in Q$ for which $C_q$ contains fewer than $2dM$ points of $P$.
The curves in $\cvs_{Q_0}$ together give $O_{d,M}(|Q|)$ incidences with $P$.
Using the dual of the procedure above, we can partition $Q\backslash Q_0$ into $O_{d,M}(1)$ subsets $Q_j$ such that the incidence graph of $P$ with each multiset $\cvs_{Q_j}$ contains no $K_{2dM,2}$.
It follows that each $\cvs_{Q_j}$ is a \emph{set}, 
because if $C_q=C_{q'}$ for two $q,q'\in Q_j$, then the fact that $q,q'\not\in Q_0$ implies that $C_q = C_{q'}$ contains at least $2dM$ points of $P$,
which would give a $K_{2dM,2}$ in the incidence graph of $P$ and $C_{Q_j}$.
Hence, each $Q_j$ is a set, and each of the $O_{d,M}(1)$ pairs $P_i,\cvs_{Q_j}$ defines an incidence graph without $K_{2,2dM}$ or $K_{2dM,2}$.
This proves the existence of the claimed partition.
\end{proof}

We now have all the pieces in place to prove Theorem \ref{thm:dimthree}.

\begin{proof}[Proof of Theorem \ref{thm:dimthree}.]
If $X$ has components of dimension less than three, we apply Theorem \ref{thm:dimtwo} to them, which gives a bound better than that in Theorem \ref{thm:dimthree}, unless one of these components is Cartesian.
As noted in Definition \ref{def:cartesian}, 
if any component of $X$ is Cartesian, then so is $X$.

Therefore, we can assume that every component of $X$ is three-dimensional, 
so that we can write $X=Z(F)$ (see for instance \cite[Theorem 1.21]{Shaf} for the fact that any variety of codimension one is the zero set of a single polynomial). 
If $F$ is chosen to have the minimum degree with this property, then it follows that $\deg(F)\leq d$.
If $F$ is not Cartesian, then by Theorem \ref{thm:null2}, 
there are no $I,J\subset\C^2$ with $|I|,|J|>d^2$ such that $I\times J\subset Z(F)$.
Setting $M= d^2+1$, it follows that the incidence graph of $P$ and $\cvs_Q$ contains no $K_{M,M}$.
Corollary \ref{cor:noKMM} then gives the bound in Theorem \ref{thm:dimthree}.
\end{proof}


\section{Repeated and distinct values of polynomial maps}
\label{sec:repeateddistinct}

We now give our applications of Theorems \ref{thm:dimtwo} and \ref{thm:dimthree} to value sets of polynomial maps.
We first consider polynomial maps $\C^4\to \C$ defined by one polynomial, and then polynomial maps $\C^4\to\C^2$ defined by two polynomials.
We could also consider maps defined by more polynomials, but this seems less interesting.

\subsection{Repeated values of polynomials}
\label{subsec:repeatedpoly}
Erd\H os \cite{Er} asked for the maximum number of unit distances that can occur between the points in a finite set $P\subset \R^2$. 
The best known upper bound is $O(|P|^{4/3})$, due to Spencer, Szemer\'edi, and Trotter \cite{SST}.
We can rephrase this statement as follows.
Let
\[F(x,y,s,t) = (x-s)^2 + (y-t)^2,\]
so that $F(x,y,s,t)$ equals the squared Euclidean distance between $(x,y)$ and $(s,t)$.
Then the number of pairs in  $P\times P\subset \R^2\times \R^2$ on which $F$ can take the value $1$ is $O(|P|^{4/3})$.
Of course, the same statement holds for any value other than $1$.

Theorem \ref{thm:dimthree} tells us to which other polynomials this bound can be generalized.
Compare this statement with Corollary \ref{cor:repvaloncurves}, 
where a sharper bound is proved for a point set contained in an algebraic curve of bounded degree.

\begin{corollary}\label{cor:repeated}
Let $F\in \C[x,y,s,t]$, let $P\subset \C^2$ be a finite set, and let $a\in \C$.
Then the number of times $F$ takes the value $a$ on $P$ satisfies
\[|\{(x,y,s,t)\in P\times P: F(x,y,s,t)=a\}| = O_{d,\eps}(|P|^{4/3+\eps}),\]
unless the polynomial $F-a$ is Cartesian. 
Over $\R$, the $\eps$ can be omitted.
\end{corollary}

The condition that $F-a$ is not Cartesian is necessary,
since if $F-a$ is $(G,K)$-Cartesian, 
then we can take half of the points of $P$ in $Z(G)$ and the other half in $Z(K)$,
to get $\Omega(|P|^2)$ repeated values.
As in Theorem \ref{thm:dimthree}, the bound is best possible in general.
That the bound is best possible for certain ``distances'' was already observed by Valtr \cite{V}, who constructed a strictly convex norm, based on the polynomial $F = (x-s)^2+y-t$, for which the upper bound $O(|P|^{4/3})$ is tight.
On the other hand, it is a major open problem in discrete geometry to show that for the Euclidean distance, this bound is not tight.

\subsection{Distinct values of polynomials}
Erd\H os \cite{Er} also asked for the minimum number of \emph{distinct} distances determined by a finite set $P\subset \R^2$.
Guth and Katz \cite{GK} proved an almost tight bound $\Omega(|P|/\log|P|)$.
It is natural to ask the same question for other functions of pairs of points in the plane.
Roche-Newton and Rudnev \cite{RNR} proved the same lower bound for the ``Minkowski distance'' $(x-s)^2 - (y-t)^2$, and Rudnev and Selig \cite{RS} extended it to other quadratic distances.
Garibaldi, Iosevich, and Senger \cite[Chapter 9]{GIS} considered the dot product $xs+yt$, but only proved the bound $\Omega(|P|^{2/3})$ for the number of distinct dot products determined by $P$; 
improving this bound is an interesting open problem.
Iosevich, Roche-Newton, and Rudnev \cite{IRR} proved the slightly better bound $\Omega(|P|^{9/13})$ for the number of distinct values of $xt-ys$.
We are not aware of other such polynomials for which any bound has been proved.

As a corollary of Corollary \ref{cor:repeated}, we obtain a bound on the number of distinct values of a polynomial, for any polynomial for which it could hold.
The bound is basically $\Omega(|P|^{2/3})$, 
so relatively weak, 
but we do not know any more elementary proof of such a bound for general polynomials.
Any such proof would have to deal with the fact that there are exceptional polynomials that are related to the Cartesian property, so techniques like ours seem to be necessary.

\begin{corollary}\label{cor:distinct}
Let $F\in \C[x,y,s,t]$ and let $P\subset \C^2$ be a finite set.
Then the number of distinct values of $F$ on $P$ satisfies
\[|F(P\times P)| = \Omega_{d,\eps}(|P|^{2/3-\eps}),\]
unless there exists an $a\in \C$ such that $F-a$ is Cartesian. 
Over $\R$, the $\eps$ can be omitted.
\end{corollary}
\begin{proof}
If every value of $F$ is repeated at most $T$ times, then the $|P|^2$ pairs in $P\times P$ must give at least $|P|^2/T$ distinct values.
For every $a\in F(P\times P)$, the fact that $F-a$ is not Cartesian lets us apply Corollary \ref{cor:repeated} to conclude that $F$ has the value $a$ for $O_{d,\eps}(|P|^{4/3+\eps})$ pairs in $P\times P$.
The stated bound follows.
\end{proof}

Note that if $F-a$ is $(G,G)$-Cartesian for $a\in \C$, 
then $P\subset Z(G)$ gives $F(P\times P) = \{a\}$, so there cannot be a good lower bound on the number of distinct values of $F$.
However, if $F-a$ is $(G,K)$-Cartesian with $G\neq K$, this construction does not work.
Thus the right exception in Corollary \ref{cor:distinct} would probably be ``unless there exist $a\in \C$ and $G\in \C[x,y]$ such that $F-a$ is $(G,G)$-Cartesian''; unfortunately, this does not quite seem to follow from our proof.

The proof of Corollary \ref{cor:distinct} gives a somewhat stronger statement: There exists $p\in P$ such that $|F(p\times P)| = \Omega_{d,\eps}(|P|^{2/3-\eps})$, unless some $F-a$ is Cartesian.

\subsection{Repeated values of polynomial maps}

Just as Theorem \ref{thm:dimthree} gives us bounds on the number of repeated and distinct values of a single polynomial, 
Theorem \ref{thm:dimtwo} can give us such bounds for polynomial maps defined by two polynomials.
We focus on the case of polynomial maps $\mathcal{F}:\C^4\to\C^2$ of the form
\[\mathcal{F}(x,y,s,t)= (F_1(x,y,s,t), F_2(x,y,s,t))\]
with $F_1,F_2\in \C[x,y,s,t]$.
Note that for a finite set $P\subset \C^2$
 and a point $(a,b)\in\C^2$, 
\[Z(F_1-a,F_2-b)\cap (P\times P)\]
 is the set of pairs $(x,y),(s,t)\in P$ for which $\mathcal{F}(x,y,s,t) = (a,b)$.

\begin{corollary}\label{cor:maprepeated}
Let $F_1,F_2\in \C[x,y,s,t]$ be two polynomials of degree at most $d$, and let $\mathcal{F}:\C^4\to\C^2$ be the polynomial map defined by $\mathcal{F}(x,y,s,t)= (F_1(x,y,s,t), F_2(x,y,s,t))$.
For $P\subset \C^2$ and $(a,b)\in \C^2$, we have
\[|\{(x,y,s,t)\in P\times P: \mathcal{F}(x,y,s,t)=(a,b)\}| = O_d(|P|),\]
unless $F_1-a$ and $F_2-b$ have a common factor, or there are $G\in \C[x,y]\backslash\C,K\in\C[s,t]\backslash\C$ such that $F_1-a$ and $F_2-b$ are both $(G,K)$-Cartesian. 
\end{corollary}
\begin{proof}
If $F_1-a$ and $F_2-b$ do not have a common factor, 
then $Z(F_1-a,F_2-b)$ has dimension two, so that we can apply Theorem \ref{thm:dimtwo} to $X = Z(F_1-a,F_2-b)$.
\end{proof}

Both exceptions are necessary.
As before, if $F_1-a$ and $F_2-b$ are both $(G,K)$-Cartesian, then we can get $|Z(F_1-a,F_2-b)\cap (P\times P)| = \Omega(|P|^2)$.
If $F_1-a$ and $F_2-b$ have a common factor $F_3$, 
then $|Z(F_1-a,F_2-b)\cap (P\times P)| \geq |Z(F_3)\cap (P\times P)|$, and we know from the constructions in Section \ref{sec:intro} that there are polynomials $F_3$ for which we have $|Z(F_3)\cap (P\times P)| = \Omega(|P|^{4/3})$.
Of course, Theorem \ref{thm:dimthree} could in this case be used to get the upper bound $O_{d,\eps}(|P|^{4/3+\eps})$.

\subsection{Distinct values of polynomial maps}
Finally, we consider distinct values of polynomial maps $\C^4\to \C^2$.
We say that $F_1$ and $F_2$ are \emph{inner equivalent} if there exist $\varphi_1,\varphi_2\in \C[z]$ and $\psi\in\C[x,y,s,t]$ such that $F_1 = \varphi_1\circ \psi$ and $F_2 = \varphi_2\circ \psi$.
Note that if $F_1$ and $F_2$ are inner equivalent, then the map $\mathcal{F} = (F_1,F_2):\C^4\to \C^2$ factors into $\psi:\C^4\to \C$ and $(\varphi_1,\varphi_2):\C\to \C^2$.
In this case, the number of distinct values of $\mathcal{F}$ is essentially the same as that of $\psi$, and thus Corollary \ref{cor:distinct} is the relevant statement.

\begin{corollary}\label{cor:mapdistinct}
Let $F_1,F_2\in \C[x,y,s,t]$ be two polynomials of degree at most $d$, and let $\mathcal{F}:\C^4\to\C^2$ be the polynomial map defined by $\mathcal{F}(x,y,s,t)= (F_1(x,y,s,t), F_2(x,y,s,t))$.
For $P\subset \C^2$, we have 
\[|\mathcal{F}(P\times P)| = \Omega_d(|P|),\]
unless $F_1$ and $F_2$ are inner equivalent, 
or there is an $(a,b)\in \C^2$ such that $F_1-a$ and $F_2-b$ are both $(G,K)$-Cartesian.
\end{corollary}
\begin{proof}
We use a generalization of Stein's theorem \cite{St} to more variables, a proof of which can be found for instance in \cite{B}.
It states that for a non-composite polynomial $F$ over $\C$ in any number of variables, the number of $\lambda\in \C$ for which the polynomial $F-\lambda$ is reducible is at most $2\deg(F)^2$ (a polynomial is \emph{composite} if there is a nonlinear univariate polynomial $\varphi$ and a polynomial $\psi$ such that $F = \varphi\circ \psi$).

We can assume that $F_1,F_2$ are non-composite, because if $F_1 = \varphi\circ \psi$, then $|F_1(P\times P)|\geq \frac{1}{d}|\psi(P\times P)|$, and similarly for $F_2$, so the number of distinct values of $\mathcal{F}$ is asymptotically unchanged by removing the outside functions.
Thus there are sets $S_1,S_2\subset \C$ of size at most $2d^2$, such that for $a\not\in S_1$, $F_1-a$ is irreducible, and for $b\not\in S_2$, $F_2-b$ is irreducible.

Let $f(d)$ be the function such that the upper bound in Corollary \ref{cor:maprepeated} is $f(d)|P|$.
We can assume that $|F_1(P\times P)|$ and $|F_2(P\times P)|$ are less than $|P|/(16d^2f(d))$, since otherwise we are done.

Consider $a\in S_1$.
By Corollary \ref{cor:repeated}, 
there are $O_{d,\eps}(|P|^{4/3+\eps})$ pairs $(p,q)\in P\times P$ for which $F_1(p,q)=a$.
Now consider $b\in  F_2(P\times P)$.
Let $F_3$ be the greatest common divisor of $F_1-a$ and $F_2-b$, and set $\widetilde{F}_1 = (F_1-a)/F_3$ and $\widetilde{F}_2=(F_2-b)/F_3$.
The pairs $(p,q)\in P\times P$ for which $F_3(p,q)=0$ have been counted among the ones with $F_1(p,q)=a$, so it remains to count the ones with $\widetilde{F}_1(p,q)=0$ and $\widetilde{F}_2(p,q)=0$.
If one of $\widetilde{F}_1 $ and $\widetilde{F}_2$ is constant, then the variety $Z(\widetilde{F}_1,\widetilde{F}_2)$ is empty.
Otherwise, $\widetilde{F}_1 $ and $\widetilde{F}_2$ are coprime and $Z(\widetilde{F}_1,\widetilde{F}_2)$ has dimension at most two,  
so by Corollary \ref{cor:maprepeated} there are $O_d(|P|)$ pairs $(p,q)\in P\times P$ for which $\widetilde{F}_1(p,q) =0$ and $\widetilde{F}_2(p,q)=0$.

Thus, 
for a fixed $a\in S_1$, we count 
\[O_{d,\eps}(|P|^{4/3+\eps}) + |F_2(P\times P)|\cdot f(d)|P| \leq \frac{1}{8d^2}|P|^2\]
pairs $(p,q)\in P\times P$ for which $\mathcal{F}(p,q) = (a,b)$ for some $b\in F_2(P\times P)$.
Summing over all $a\in S_1$, 
we have fewer than $\frac{1}{4}|P|^2$ pairs $(p,q)\in P\times P$ for which $\mathcal{F}(p,q) = (a,b)$ for some $(a,b)\in S_1\times F_2(P\times P)$.
By a symmetric argument, 
we have fewer than $\frac{1}{4}|P|^2$ pairs $(p,q)\in P\times P$ for which $\mathcal{F}(p,q) = (a,b)$ for some $(a,b)\in F_1(P\times P)\times S_2$.
Altogether, there are fewer than $\frac{1}{2}|P|^2$ pairs for which $\mathcal{F}(p,q)\in S_1\times S_2$.
Consequently, there are at least $\frac{1}{2}|P|^2$ pairs $(p,q)\in P\times P$ for which $\mathcal{F}(p,q) = (a,b)$ with $a\not\in S_1$ and $b\not\in S_2$.

Now consider $a\not\in S_1$ and $b\not\in S_2$.
We have $F_1-a$ and $F_2-b$ both irreducible, which implies that they do not have a common factor, 
unless $F_1 = \alpha F_2 + \beta$ for some $\alpha,\beta\in \C$.
But that would mean that $F_1$ and $F_2$ are inner equivalent (moreover, they were inner equivalent before we removed any outside functions).
Thus $F_1-a$ and $F_2-b$ do not have a common factor, and $Z(F_1-a,F_2-b)$ has dimension two.
By Corollary \ref{cor:maprepeated},
there are $O_d(|P|)$ pairs $(p,q)\in P\times P$ such that $\mathcal{F}(p,q) = (a,b)$, unless $F_1-a$ and $F_2-b$ are both $(G,K)$-Cartesian for some $G,K$.
Since there are at least $\frac{1}{2}|P|^2$ pairs $(p,q)\in P\times P$ for which $\mathcal{F}(p,q)=(a,b)$ with $a\not\in S_1$ and $b\not\in S_2$,
there must be $\Omega_d(|P|)$ distinct $(a,b)\in \mathcal{F}(P\times P)$.
\end{proof}

There are maps for which the bound in Corollary \ref{cor:mapdistinct} is tight.
Consider the vector addition map $\mathcal{F}(p,q) = p+ q = (x+s,y+t)$. 
Clearly, if $P$ is a set of points in arithmetic progression (and thus on a line), then $|\mathcal{F}(P\times P)| = O(|P|)$.
The polynomials $F_1 = x+s$ and $F_2 = y+t$ are not inner equivalent, and $F_1-a$ and $F_2-b$ are never both $(G,K)$-Cartesian (both are Cartesian, but with different $G,K$).



\section{Discussion}
\label{sec:discussion}

\paragraph{Tightness.}
The main open question regarding Theorem \ref{thm:dimthree} is for which varieties the bound is tight, and for which it can be improved.
As mentioned in the Introduction, the bound is tight for various varieties, while for others it is conjectured that improvement is possible.

\begin{question}
For which non-Cartesian polynomials $F\in \C[x,y,s,t]$ and sets $P,Q$ of size $n$ is the bound $|Z(F)\cap (P\times Q)| = O(n^{4/3})$ tight, 
and for which is it not?
\end{question}

Of course, this is just a generalization of the well-known unit distance problem,
but this view does raise some new possibilities.
It may be possible to find a polynomial for which the bound can be improved more easily than for the Euclidean distance polynomial.
It would also be interesting to find and classify more polynomials for which the bound is tight.

\paragraph{About the $\eps$.}
Note that Theorem \ref{thm:shefferzahl} is conjectured to hold without the $\eps$, which would imply the same for our main theorem.
Zahl \cite{Z} did prove such a bound without $\eps$, but with more restrictive conditions.
Specifically, his theorem requires that the incidences occur at  transversal intersection points of the curves, and that the curves are smooth.

We believe we can handle the first condition in our situation, because for our curves common tangent lines should be relatively rare.
We could handle the second condition if we assume $Z(F)$ to be nonsingular, because then only a relatively small subset of our curves could be singular.
However, if $Z(F)$ is singular, then it could be that all our curves are singular, in which case we do not see a way to apply Zahl's theorem.

\paragraph{About the Cartesian form.}
We currently do not have a good method or algorithm for proving that a polynomial is not Cartesian, and we do not know any instance where our algebraic condition (that $F$ is not Cartesian) is easier to establish than the combinatorial condition of Theorem \ref{thm:shefferzahl} (that the incidence graph does not contain certain complete bipartite subgraphs).

\begin{question}
Is there an algorithm that determines if a polynomial $F\in \C[x,y,s,t]$ is Cartesian, i.e., if it can be written as $F(x,y,s,t) = G(x,y)H(x,y,s,t) + K(s,t)L(x,y,s,t)$ with $G$ and $K$ not constant?
\end{question}

\paragraph{Related results.}
Theorem \ref{thm:dimthree} can be thought of as a variant of the Pach-Sharir incidence bound, where a combinatorial condition is replaced by an algebraic one.
Specifically, Theorem \ref{thm:shefferzahl} gives a bound on $I(\pts,\cvs)$ if the incidence graph of $\pts$ and $\cvs$ does not contain a $K_{2,M}$ or $K_{M,2}$.
In most algebraic applications of this theorem, there is a polynomial $F\in \R[x,y,s,t]$ such that the curves in $\cvs$ are of the form $\{p:F(p,q_i)=0\}$ for fixed $q_i$ in some finite set $Q$; see for instance \cite{RSS, SSS,Sh}.
In such applications, our theorem replaces the combinatorial condition on the incidence graph by the purely algebraic condition that $F$ is not Cartesian.

Yet another way to view Theorem \ref{thm:dimthree} is as follows.
An \emph{algebraic graph} on $\R^2\times \R^2$ is a bipartite graph whose two parts are finite sets $P,Q\subset \R^2$, with edges defined by a polynomial $F\in \R[x,y,s,t]$, i.e., there is an edge between $p\in P$ and $q\in Q$ if and only if $F(p,q)=0$.
Fox et al. \cite{FPSSZ} proved that if an algebraic graph on $\R^2\times \R^2$ contains no $K_{t,t}$, then the number of edges is bounded by $O_t(|P|^{2/3}|Q|^{2/3}+|P|+|Q|)$ (in fact, they proved much more general statements for semialgebraic graphs with vertex sets in $\R^d$).
Again, Theorem \ref{thm:dimthree} states the same bound, but replaces the combinatorial condition by an algebraic condition.


\paragraph{Possible extensions.}
In this paper, we have intentionally not looked beyond $\C^4$, but there are natural questions to ask for other Cartesian products. We mention a few, without going into too much detail.

We can consider varieties $X\subset\C^D\times \C^E$ and try to bound the intersection $|X\cap (P\times Q)|$ for $P\subset \C^D, Q\subset \C^E$.
There will be exceptions for $X$ defined by ``Cartesian'' polynomials of the form $GH+KL$, with $G$ a polynomial in the first $D$ variables and $K$ a polynomial in the last $E$ variables.
In analogy with Theorem \ref{thm:dimtwo}, we would expect that in the case $D=E=\dim(X)$ the bound will be $O(|P|+|Q|)$.
For $X$ of intermediate dimension, we would expect intermediate exponents, probably related to those in higher-dimensional incidence theorems.

In another direction, we can take $X\subset (\C^2)^k$ and consider $|X\cap \prod_{i=1}^k P_i|$ for $P_i\subset \C^2$.
There will be exceptional varieties $X$ defined by polynomials of the form $\sum_{i=1}^k G_iH_i$, with each $G_i$ a function of the $i$-th pair of variables.
In analogy with the general Schwartz-Zippel lemma, 
we expect that if $\dim(X)= 2\ell$, then the bound will be $O(|P|^\ell)$.

Of course one can combine these two directions and look at arbitrary products $\prod \C^{D_i}$.
All such results would come with corresponding corollaries for value sets of polynomial maps.

Another direction to proceed in is to ask the same questions for other fields, in particular finite fields.
Our proof of Theorem \ref{thm:dimtwo} seems to carry through for any algebraically closed field, but we have restricted ourselves to $\C$ for simplicity.
Theorem \ref{thm:dimthree}, however, relies on the incidence bound in Theorem \ref{thm:shefferzahl}, 
which does not yet have any analogue over finite fields (at least not with the same algebraic flexibility).




\bibliographystyle{amsplain}

\begin{dajauthors}
\begin{authorinfo}[hnm]
  Hossein Nassajian Mojarrad\\
  \'Ecole Polytechnique F\'ed\'erale de Lausanne\\
  Lausanne, Switzerland\\
  hossein\imagedot{}mojarrad\imageat{}epfl\imagedot{}ch \\
  \url{https://dcg.epfl.ch/hossein}
\end{authorinfo}
\begin{authorinfo}[tp]
  Thang Pham\\
  \'Ecole Polytechnique F\'ed\'erale de Lausanne\\
  Lausanne, Switzerland\\
  phamanhthang\imagedot{}vnu\imageat{}gmail\imagedot{}com \\
  \url{https://dcg.epfl.ch/thang}
\end{authorinfo}
\begin{authorinfo}[laci]
  Claudiu Valculescu\\
  \'Ecole Polytechnique F\'ed\'erale de Lausanne\\
  Lausanne, Switzerland\\
  claudio040788\imageat{}gmail\imagedot{}com \\
  \url{https://dcg.epfl.ch/page-100304-en.html}
\end{authorinfo}
\begin{authorinfo}[andy]
  Frank de Zeeuw\\
  \'Ecole Polytechnique F\'ed\'erale de Lausanne\\
  Lausanne, Switzerland\\
  fdezeeuw\imageat{}gmail\imagedot{}com\\
  \url{https://dcg.epfl.ch/page-84876-en.html}
\end{authorinfo}
\end{dajauthors}

\end{document}